\documentclass[10pt,reqno]{amsart}
\usepackage{amsmath,amscd,amsfonts,amsthm,amssymb,latexsym,mathrsfs,graphicx,url}
\usepackage{hyperref}
\hypersetup{
    bookmarks=true,         
    unicode=false,          
    pdftoolbar=true,        
    pdfmenubar=true,        
    pdffitwindow=false,     
    pdfstartview={FitH},    
    pdftitle={My title},    
    pdfauthor={Author},     
    pdfsubject={Subject},   
    pdfcreator={Creator},   
    pdfproducer={Producer}, 
    pdfkeywords={keyword1} {key2} {key3}, 
    pdfnewwindow=true,      
    colorlinks=true,       
    linkcolor=blue,          
    citecolor=blue,        
    filecolor=blue,      
    urlcolor=blue      
    }

\theoremstyle{plain}
   \newtheorem{theorem}{Theorem}[section]
   
   \newtheorem{lemma}[theorem]{Lemma}

   \newtheorem{conjecture}[theorem]{Conjecture}
   \newtheorem{question}{Question}
\theoremstyle{definition}
   
   \newtheorem{example}{Example}[section]

\theoremstyle{remark}
   \newtheorem{remark}[theorem]{Remark}

\numberwithin{equation}{section}

\def\kk{\kern.2ex\mbox{\raise.5ex\hbox{{\rule{.35em}{.12ex}}}}\kern.2ex}

\newcommand{\LP}{\mathscr{L{\kk}P}}
\newcommand{\NN}{\mathbb{N}}

\newcommand{\zz}{\mathbf{z}}

\newcommand{\xx}{\mathbf{x}}

\newcommand{\A}{\mathcal{A}}
\newcommand{\EE}{\mathcal{E}}

\newcommand{\LL}{\mathcal{L}}

\newcommand{\RR}{\mathbb{R}}
\newcommand{\CC}{\mathbb{C}}

\newcommand{\PPP}{\mathscr{P}}

\renewcommand{\Im}{{\rm Im}}
\renewcommand{\Re}{{\rm Re}}
\newcommand{\PF}{{\rm PF}}
\title[Infinite log-concavity for polynomial PF]{Infinite log-concavity for polynomial \\ P\'olya frequency sequences} 

\author{Petter Br\"and\'en}
\thanks{The first author is a Royal Swedish Academy of Sciences Research Fellow
  supported by a grant from the Knut and Alice Wallenberg
  Foundation. The research is also supported by the G\"oran Gustafsson Foundation.}
\author{Matthew Chasse}

\address{Department of Mathematics, Royal Institute of Technology, SE-100 44 Stockholm,
Sweden}
\email{pbranden@kth.se, chasse@kth.se}

\begin{document}
\begin{abstract}
McNamara and Sagan conjectured that if $a_0,a_1, a_2, \ldots$ is a P\'olya frequency (PF) sequence, then so is 
$a_0^2, a_1^2 -a_0a_2, a_2^2-a_1a_3, \ldots$. We prove this conjecture for a natural class of PF-sequences which are interpolated by polynomials. In particular, this proves that the columns of  Pascal's triangle are infinitely log-concave, as conjectured by McNamara and Sagan. 
We also give counterexamples to the first mentioned conjecture. 

Our methods provide families of nonlinear operators that preserve the property of having only real and non-positive zeros.
\end{abstract}
\maketitle

\renewcommand\theenumi{\roman{enumi}}
\thispagestyle{empty}
\section{Introduction}
Consider the operator $\LL$  defined on sequences of real numbers by
$$
\LL(A)= \{a_0^2, a_1^2 -a_0a_2, a_2^2-a_1a_3, a_3^2-a_2a_4, \ldots\}, 
$$
where  $A=\{a_0,a_1,a_2, a_3,\ldots\}$. 
Hence a sequence $A$ is log--concave if and only if $\LL(A)$ is a nonnegative sequence. A sequence $A$ is said to be $k$-\emph{fold log-concave} if $\LL^j(A)$ is a nonnegative sequence for all $0\leq j \leq k$, and \emph{infinitely log-concave} if it is $k$-fold log-concave for all $k \geq 0$.  The work of Boros and Moll \cite[p.~157]{BM} spurred the interest in infinite log-concavity in the combinatorics community, although more than a decade earlier similar notions were introduced by Craven and Csordas \cite{CC1,CC2} in the theory of entire functions.  Boros and Moll showed that the sequences $\{d_\ell(m)\}_{\ell=0}^m$, $m\in\NN$, are unimodal, where
\[
d_\ell(m) = 2^{-2m} \sum_{k=\ell}^m 2^k \binom{2m-2k}{m-k} \binom{m+k}{m} \binom{k}{\ell},
\]
and conjectured furthermore that $\{d_\ell(m)\}_{\ell=0}^m$ is infinitely log-concave for each $m\in\NN$.  Chen \emph{et.~al.} \cite{Chen} proved $3$-fold log-concavity of $\{d_\ell(m)\}_{\ell=0}^m$ by proving a related conjecture of the first author \cite{B} which implies $3$-fold log-concavity of $\{d_\ell(m)\}_{\ell=0}^m$ for each $m \in \mathbb{N}$ by the work of Craven and Csordas \cite{CC2}. Concurrent with their conjecture for $\{d_\ell(m)\}_{\ell=0}^m$, Boros and Moll suggested proving the infinite log-concavity of the more fundamental sequence of binomial numbers, $\{\binom{n}{k}\}_{k\ge0}$, for any $n\in\NN$ (established by the first author in \cite{B}).  The following additional conjectures for the columns and diagonals of Pascal's triangle, the first of which has remained open, were made McNamara and Sagan.

\begin{conjecture}[{\cite[Conjecture 4.1]{MS}}]\label{binom-column}
The sequence $\{\binom{n+k}{k}\}_{n\ge0}$ is infinitely log-concave for all fixed $k\ge 0.$ 
\end{conjecture}  

\begin{conjecture}[{\cite[Conjecture 4.4]{MS}}]\label{binom-diag}
Suppose that $u$ and $v$ are distinct nonnegative integers.  Then $\{\binom{n+mu}{mv}\}_{m\ge0}$ is infinitely log-concave for all $n\ge 0$ if and only if $u<v$ or $v=0$. 
\end{conjecture}  
\noindent
Conjecture \ref{binom-diag} was established by results in \cite{B} and \cite{Yu}. 
We settle Conjecture\ref{binom-column} in the affirmative by proving Theorem \ref{main}.   Part of our proof generalizes to provide examples of nonlinear transformations of polynomials which preserve the property that the zeros lie in a prescribed subset of $\CC$.  Nonlinear transformations of polynomials have been the topic of several recent investigations \cite{B,G,Karp}.

A (possibly infinite) matrix is called \emph{totally nonnegative} if all of its minors are nonnegative.  In particular, if the Toeplitz matrix $M=(a_{j-i})_{i,j\ge 0}$ defined by the sequence $\{a_k\}_{k=0}^\infty$ is totally nonnegative, then then the sequence is called totally nonnegative or said to be a {\em P\'olya frequency sequence} ($\PF$-sequence).  Let $\PF$ be the class of all P\'olya frequency sequences, which is characterized by the following theorem.

\begin{theorem}[\cite{ASW,E}]\label{gene}
Let $\{a_k\}_{k=0}^\infty$ be a sequence of real numbers. Then $\{a_k\}_{k=0}^\infty$ is a P\'olya frequency sequence if and only if its power series generating function is analytic in a neighborhood of the origin and has the expansion
\begin{equation}\label{pf-gen}
\sum_{k=0}^\infty a_kx^k = cx^me^{\gamma x} {\prod_{k\ge 0} (1+\alpha_k x) }\Big / { \prod_{k\ge 0} (1-\beta_k x)}
\end{equation}
where $c,\gamma\ge 0$, $\alpha_k,\beta_k>0$, $m\in\NN\cup\{0\}$, and $\sum_{k\ge 0} (\alpha_k+\beta_k)<\infty$.
\end{theorem}

The first author \cite{B} proved that if the sequence $A=\{a_k\}_{k=0}^\infty$ is the sequence of coefficients of an entire function with only real non-positive zeros, then so is $\LL(A)$ (the case of no poles in \eqref{pf-gen}), resolving a conjecture of Fisk, McNamara-Sagan, and Stanley.  For P\'olya frequency sequences the following conjecture of McNamara-Sagan has remained open.

\begin{conjecture}[{\cite[Conjecture 7.4]{MS}}]\label{McNamara-Sagan}
Let $A=\{a_k\}_{k=0}^\infty$ be a sequence of real numbers. If $A$ is  a  P\'olya frequency sequence, then $\LL(A)$ is also a  P\'olya frequency sequence.
\end{conjecture}

Simple examples show that Conjecture \ref{McNamara-Sagan} is false in general (consider the sequence generated by $(1+x)^3/(1-10x)$).  Furthermore, the following example shows that some $\PF$-sequence are not even $3$-fold log-concave.  

\begin{example}\label{not-inf-lc}
Let $A=\{a_k\}_{k=0}^\infty=\{1,8,35,116,332,\dots\}$ be the sequence of real numbers with the rational generating function
\[
\frac{(1+x)^3}{(1-x)(1-2x)^2} = 1 + 8x + 116x^2 + 332x^3 + \dots.
\]
Then,
\begin{align*}
\LL^4(A)=\{1, &67251334144, 681452113625701425, 30700964335097660866560,\\
 &-41699291012783844888674304, \dots \},
\end{align*}
so $\LL^3(A)$ is not log-concave.
\end{example}

Wagner \cite{W} has shown that $\PF$--sequences which are interpolated by polynomials are closed under the Hadamard product, while arbitrary $\PF$--sequences are not.  This motivates us to look for a subclass of the $\PF$-sequences which is closed under $\LL$.  Such a subclass cannot include all polynomial $\PF$-sequences, as we observe in Example \ref{poly-ce}. 
 We show in Section \ref{main-results} that Conjecture \ref{McNamara-Sagan} holds for several large classes of $\PF$--sequences which are interpolated by polynomials.  These classes are natural in a sense, being indicated by Wagner's results for the Hadamard product with an additional requirement (cf. Theorem \ref{main}).  In Section \ref{extensions}, we generalize the proof in Section \ref{main-results} to provide some new nonlinear transformations which have prescribed effects on the zero loci of polynomials. 

\section{Main results}\label{main-results}
We are interested in $\PF$--sequences that are interpolated by polynomials, i.e., sequences $A= \{a_k\}_{k=0}^\infty \in \PF$ where $a_k=p(k)$ for all $k \in \NN$, for some  $p \in \RR[x]$. In terms of generating functions this is equivalent to (by Theorem \ref{gene}) that the generating function of $A$ is of the form 
$$
\sum_{k=0}^\infty a_k x^k = \frac {w(x)}{(1-x)^d},
$$
where $w(x)$ is a real--rooted polynomial with nonnegative coefficients. 
However we will also require that $\LL^k(A)$ is interpolated by polynomials for all $k \in \NN$. We shall see that the following  are natural such classes 
\begin{align*}
\A_{0} &= \big\{ \{p(k)\}_{k=0}^\infty \in \PF: p \in \RR[x] \mbox{ and } p(0)=p(1) =0  \big\}, \\
\A_{-1} &= \big\{ \{p(k)\}_{k=0}^\infty \in \PF: p \in \RR[x] \mbox{ and } p(-1)=p(0) =0  \big\},   \mbox{ and }\\
\A_{-2} &= \big\{ \{p(k)\}_{k=0}^\infty \in \PF: p \in \RR[x] \mbox{ and } p(-2)=p(-1) =0  \big\}. 
\end{align*}

Our main result is the following.

\begin{theorem}\label{main}
Let $A=\{a_j\}_{j=0}^\infty$ be a sequence of real numbers. The following conditions on $A$ are equivalent. 
\begin{enumerate}
\item For each $k \in \NN$, $\LL^k(A)$ is a P\'olya frequency sequence which is interpolated by a polynomial. 
\item $A$, $\LL(A)$  and $\LL^2(A)$ are interpolated by polynomials, and $A$ as well as $\LL(A)$ are P\'olya frequency sequences. 
\item $A \in  \A_{0}\cup \A_{-1} \cup \A_{-2}$.
\end{enumerate}
Moreover if $A \in \A_i$ where  $i \in \{-2,-1,0\}$, then $\LL(A) \in  \A_i$. 
\end{theorem}
Conjecture \ref{binom-column} trivially holds for $k=0,1$, see \cite{MS}. For $k \geq 2$, $\{ \binom{n+k}{k} \} _{n\ge0} \in \A_{-2}$, and in particular $\{\binom{n+k}{k}\}_{n\ge0}$ is infinitely log-concave, by Theorem \ref{main}. This solves Conjecture \ref{binom-column}.

\begin{example}\label{poly-ce}
Let $A=\{p(k)\}_{k=0}^\infty$. By Lemma \ref{wgn}, a sequence interpolated by a polynomial $p\in\RR[k]$ is a P\'olya frequency sequence if and only if  $\EE(p)$ is $[-1,0]$-rooted.  Thus, for $p(x)=x^3$, $\EE(p)$ is $[-1,0]$-rooted and in addition $p(0)=0$.  But 
\begin{align*}
\EE&((p(k))^2-p(k+1)p(k-1)) \\
&= 72x^4+108x^3+36x^2+1
\end{align*}
has two non-real zeros, so $\LL(A)$ is not a P\'olya frequency sequence. Similarly, $q(x)=(x+1)(x+3)^2$ interpolates a P\'olya frequency sequence, but $\LL(\{q(k)\}_{k=0}^\infty)$ does not, providing an example where $q(-1)=0$ and $q(0)^2-q(1)q(-1) = q(0)^2$. 
\end{example}

Although Example \ref{poly-ce} shows that iterates of $\{k^d\}_{k=0}^\infty$ under  $\LL$  cannot remain in $\PF$, it is still possible that they are infinitely log-concave.

\begin{question}\label{pktod}
Is $\{k^d\}_{k=0}^\infty$ infinitely log-concave for all positive integers $d$? 
\end{question}

Question \ref{pktod} is analogous to a problem for the coefficients of entire functions posed by Craven and Csordas \cite[p.~6, Open Problem (e)]{CC2}. 

\begin{example}
A \emph{labeled poset} is a partially ordered set on $[p]:=\{1,2,\ldots, p\}$, where $p \geq 1$. 
Let $P$ be a labeled poset. A $P$-\emph{partition} is a map $f : P \rightarrow \NN$ such that 
\begin{itemize}
\item if $i <_P j$, then $f(i) \geq f(j)$, and 
\item if $i <_P j$ and $i>j$, then $f(i) > f(j)$,
\end{itemize}
where $<$ is the usual order on the integers. 
The \emph{order polynomial} of a labeled poset $P$ is defined by
$$
\Omega_P(n) = |\{f : P \rightarrow [n] : f \mbox{ is a } P\mbox{-partition}\}|.
$$
The Neggers--Stanley conjecture asserted that $\{\Omega_P(n)\}_{n=0}^\infty$ is a $\PF$-sequence for all labeled posets $P$, see e.g.  \cite{Br}. Counterexamples to the Neggers--Stanley conjecture have been found \cite{B3,Stembridge}, but it remains open in the case of \emph{naturally labeled and graded} and \emph{sign graded} posets (see \cite{Stembridge}). If $P$ is not an antichain, then it is easy to see that either 
$\Omega_P(-1)=\Omega_P(0)=0$, or $\Omega_P(0) = \Omega_P(1)=0$. Hence if $P$ is a not an antichain, and if the Neggers--Stanley conjecture holds for $P$, then $\{\Omega_P(n)\}_{n=0}^\infty$ is infinitely log-concave by Theorem \ref{main}. 
\end{example}

Let $\EE : \RR[x] \rightarrow \RR[x]$ be the linear operator defined by 
$$
\EE \binom x k = x^k, \mbox{ for all } k \in \NN, 
$$
where $\binom x 0 = 1$ and $\binom x k = x(x-1) \cdots (x+k-1)/k!$ if $k \geq 1$. 
\begin{lemma}[{\cite[Proposition 2.3]{W}}]\label{wgn}
Let $A=\{p(k)\}_{k=0}^\infty$ where $p \in \RR[x]$ has positive leading coefficient. Then $A$ is a P\'olya frequency sequence if and only if all the zeros of $\EE(p)$ are real and located in the interval $[-1,0]$. 
\end{lemma}
The \emph{diamond product} is a natural bilinear form associated to $\EE$ defined by 
$$
f\diamond g = \EE(\EE^{-1}(f)\EE^{-1}(g)).
$$
The diamond product may be expressed as 
\begin{equation}\label{dmd}
(f\diamond g)(x)= \sum_{k=0}^\infty \frac {f^{(k)}(x)} {k!} \frac {g^{(k)}(x)} {k!} x^k(x+1)^k,
\end{equation}
see \cite[Theorem 2.7]{W}. 


\begin{lemma}\label{algg}
Let $p \in \RR[x]$ be such that $p(0)=p(-1)=0$, and let $f= \EE(p)$. Then $f(0)=f(-1)=0$ and 
\begin{align}\label{gd}
&\EE\big(p(x)^2-p(x-1)p(x+1)\big) = \nonumber \\
&x(1+x)\sum_{k=0}^\infty \left(\frac {g^{(k)}(x)} {k!} \frac {g^{(k)}(x)} {k!} -\frac {g^{(k-1)}(x)} {(k-1)!} \frac {g^{(k+1)}(x)} {(k+1)!} \right) x^k(1+x)^k,
\end{align}
where $f(x)=x(1+x)g(x)$ and ${g^{(-1)}(x)}/{(-1)!} := 0$.
\end{lemma}
\begin{proof}
Since $\binom \xi k = \xi^k$ for $\xi \in \{-1,0\}$, we have  $p(0)=\EE(p)(0)$ and $p(-1) = \EE(p)(-1)$ for all $p \in \RR[x]$. Let $p \in \RR[x]$ be such that $p(0)=p(-1)=0$, and let $f= \EE(p)=x(1+x)g$. By the definition of the diamond product 
\begin{equation}\label{str}
\EE\big(p(x)^2-p(x-1)p(x+1)\big) = f\diamond f - \EE\Big(\EE^{-1}(f)(x+1)\EE^{-1}(f)(x-1)\Big).
\end{equation}
We claim that for any $h \in \RR[x]$, 
\begin{equation}\label{shift}
\left.\EE^{-1}(xh)\right|_{x=x+1}= \EE^{-1}((x+1)h) \mbox{ and } \left.\EE^{-1}((x+1)h)\right|_{x=x-1} = \EE^{-1}(xh). 
\end{equation}

The second identity of \eqref{shift} follows from the first by a shift of argument, and to prove the first it is enough to prove it for $x^k$, $k \in \NN$:
\begin{align*}
\left. \EE^{-1}(x^{k+1})\right|_{x=x+1} &= \binom {x+1} {k+1}= \binom x {k+1} + \binom x k \\
&= \EE^{-1}(x^{k+1})+ \EE^{-1}(x^k)= \EE^{-1}((x+1)x^k).
\end{align*}
By \eqref{str} and \eqref{shift},
\begin{align*}
\EE\big(p(x)^2-p(x-1)p(x+1)\big) &= (x(x+1)g)\diamond (x(x+1)g) - ((x+1)^2g)\diamond (x^2g) \\
&= (xg) \diamond (xg) - g\diamond (x^2g).
\end{align*}
Now \eqref{gd} follows by using \eqref{dmd} and the identities $(xg)^{(k)} = xg^{(k)}+ kg^{(k-1)}$ and $(x^2g)^{(k)} = x^2g^{(k)}+ 2kxg^{(k-1)} + k(k-1)g^{(k-1)}$, for all $k \geq 0$.  
\end{proof}

%
%
The final ingredient for the proof of Theorem \ref{main} is the powerful Grace--Walsh--Szeg\H{o} Theorem. A {\em circular region} is a proper subset of the complex plane that is bounded by either a circle or a straight line, and is either open or closed. A polynomial is {\em multi-affine} provided that each 
variable occurs at most to the first power.   
\begin{theorem}[Grace--Walsh--Szeg\H{o}, \cite{RS}]\label{GWS}
Let $f \in \CC[z_1, \ldots, z_n]$ be a multi-affine and symmetric polynomial, and let $K$ be a circular region. Assume that either $K$ is convex or that the degree of $f$ is $n$. For any $\zeta_1, \ldots, \zeta_n \in K$ there is a $\zeta \in K$ such that 
$
f(\zeta_1, \ldots, \zeta_n)= f(\zeta, \ldots, \zeta). 
$
\end{theorem}

\begin{proof}[Proof of Theorem \ref{main}]
We prove (i) $\Rightarrow$ (ii) $\Rightarrow$ (iii) $\Rightarrow$ (i). 

The implication (i) $\Rightarrow$ (ii) is immediate. Assume (ii) for $A= \{p(k)\}_{k=0}^\infty$, where $p \in \RR[x]$. Then $\LL(A)$ is interpolated by the polynomial 
$$q(x):= p(x)^2-p(x-1)p(x+1),$$  and hence $p(-1)p(1)=0$. If $p(1)=0$, then $p(0)=0$ since $\PF$--sequences have no internal zeros (by e.g. \eqref{pf-gen}). Thus $A \in \A_0$. 

If $p(-1)=0$, then  $$0=q(-1)q(1)= -p(-2)p(0)q(1),$$
since $\LL^2(A)$ is interpolated by a polynomial. 
If $q(1)=0$, then also $0=q(0)=p(0)^2$ by the argument above, and hence $A \in \A_{-1}$.  Finally if $p(-2)p(0) =0$, then 
$A \in \A_{-1} \cup \A_{-2}$, which proves (iii).

We prove (iii) $\Rightarrow$ (i) and the final statement of the theorem simultaneously.  If $S : \RR[x] \rightarrow \RR[x]$ is the algebra automorphism defined by $S(x)=-x-1$, then $S\circ \EE=\EE \circ S$, see \cite[Lemma 4.2]{B2}. Hence if $r(x):=(-1)^{\deg p}p(-x-1) \in \RR[x]$, then $\{p(k)\}_{k=0}^\infty \in \PF$ if and only if $\{r(k)\}_{k=0}^\infty \in \PF$, by Lemma \ref{wgn}. This provides a bijection between $\A_0$ and $\A_{-2}$. Moreover $\{p(k)\}_{k=0}^\infty  \mapsto \{p(k+1)\}_{k=0}^\infty$ is a bijection between $\A_{-1}$ and $\A_0$. Hence the proof is reduced to proving $\LL : \A_{-1} \rightarrow \A_{-1}$. By Lemma \ref{algg} and Lemma \ref{wgn} it remains to prove that the nonlinear operator $T : \RR[x] \rightarrow \RR[x]$ defined by 
$$
T(g)= \sum_{k=0}^\infty \left(\frac {g^{(k)}(x)} {k!} \frac {g^{(k)}(x)} {k!} -\frac {g^{(k-1)}(x)} {(k-1)!} \frac {g^{(k+1)}(x)} {(k+1)!} \right) x^k(1+x)^k
$$
preserves the property of having all zeros in the interval $[-1,0]$.

Suppose $g(x) = \prod_{i=1}^n (x+\theta_i)$, where $0 \leq \theta_i \leq 1$. Let $y \in \CC \setminus [-1,0]$. We shall prove that $T(g)(y) \neq 0$. We claim that  we may choose $\xi \in \CC$ such that $\xi^2= y(y+1)$ and $\Re (\xi/y) >0$. Indeed if $\zeta$ is a square root of $y(y+1)$ and $\Re(\zeta/y) = 0$, then 
$\zeta^2/y^2 = 1+1/y$ is a negative real number and thus $y \in (-1,0)$. Hence we may choose $\xi = \pm \zeta$. Next we claim that $\Re(\xi/(y+\theta)) >0$ for all 
$\theta \in [0,1]$. Indeed if $\Re(\xi/(y+\theta)) \leq 0$ for some $0\leq \theta \leq 1$, then $\Re(\xi/(y+\theta')) = 0$ for some $0\leq \theta' \leq 1$ (since $\Re (\xi/y) >0$). Hence 
$$
\frac {\xi^2}{(y+\theta')^2}=  \frac {y(y+1)}{(y+\theta')^2}  
$$
is a negative real number. However if $\lambda>0$, then the zeros of  
$$
h(x)= x(x+1)+\lambda (x+\theta')^2
$$
are real and in the interval $[-1,0]$, because $h(t) > 0$ if $t \in \RR\setminus [-1,0]$ and $h(-\theta') \leq 0$.  This contradiction ($y \in [-1,0]$) proves the claim. 

Let $\zz= (z_1,\ldots, z_n)$, where $z_i= \xi/(y+\theta_i)$ for $1 \leq i \leq n$. Thus $\Re(z_i) >0$ for all $1\leq i \leq n$. Let $e_k(\xx)$ is the $k$th elementary symmetric function in the variables $\xx =(x_1,\ldots, x_n)$. Then 
\begin{equation}\label{symm-fn-ident}
\begin{aligned}
\sum_{k=0}^n g(y)e_k(\zz)t^k &= g(y) \prod_{i=1}^n \left(1+ \frac {\xi t}{y+\theta_i}\right)\\
&= \prod_{i=1}^n(y+\theta_i + \xi t) \\
&= \sum_{k=0}^n \frac{g^{(k)}(y)}{k!} \xi^k t^k,
\end{aligned}
\end{equation}
and therefore,
\[
\xi^k\frac{g^{(k)}(y)}{k!} = g(y)e_k({\bf z}).
\]

The identity
\begin{equation}\label{beauty}
\sum_{k=0}^n (e_k(\xx)^2-e_{k-1}(\xx)e_{k+1}(\xx))= e_n(\xx)\sum_{k=0}^{\lfloor n/2 \rfloor}  C_k e_{n-2k}\left(\xx+\frac 1 \xx\right),
\end{equation}
where $C_k=\binom {2k} k /(k+1)$ is a \emph{Catalan number} and $1/\xx=(1/x_1,\ldots,1/x_n)$, was proved in \cite{B}. Hence 
$$
T(g)(y)= g(y) \xi^n \sum_{k=0}^{\lfloor n/2 \rfloor}  C_k e_{n-2k}\left(\zz+\frac 1 \zz\right).
$$
For the sake of contradiction suppose $T(g)(y)=0$. Since $\Re(z_i+1/z_i) >0$ for all $1\leq i \leq n$, the Grace--Walsh--Szeg\H{o} Theorem provides a number $\eta \in \CC$, with $\Re(\eta)>0$, such that 
$$
0=\sum_{k=0}^{\lfloor n/2 \rfloor} C_k e_{n-2k}\left(\eta, \ldots, \eta\right)= \sum_{k=0}^{\lfloor n/2 \rfloor} C_k\binom n {2k} \eta^{n-2k}=: \eta^n q_n\left(\frac 1 {\eta^2}\right).  
$$
Since $\Re(\eta)>0$, we have $1/\eta^2 \in \CC \setminus \{ x \in \RR: x\leq  0\}$. Hence, the desired contradiction follows if we can prove that all the zeros of $q_n(x)$ are real and negative. This follows from the identity 
\begin{eqnarray*}
\sum_{k=0}^{\lfloor n/2 \rfloor} C_k \binom n {2k}x^{k}(1+x)^{n-2k}&=&
\sum_{k=0}^n \frac 1 {n+1} \binom {n+1} k \binom {n+1} {k+1} x^k\\ 
&=& \frac 1 {n+1} (1-x)^nP_n^{(1,1)}\left(\frac {1+x}{1-x} \right), 
\end{eqnarray*}
where $\{P_n^{(1,1)}(x)\}_n$ are \emph{Jacobi polynomials}, see \cite[p.~254]{Ra}.  The zeros of the Jacobi polynomials $\{P_n^{(1,1)}(x)\}_n$ are located in the interval $(-1,1)$. Note that the first identity in the equation above follows 
immediately from \eqref{beauty}. 
\end{proof}

\section{Non-linear differential operators acting on polynomial spaces}\label{extensions}

Here we generalize the proof of Conjecture \ref{McNamara-Sagan} to more general statements involving non-linear differential operators, paralleling results in \cite{B}.  For our extensions we require the following theorem.  

\begin{theorem}[{\cite[Theorem 2.1]{B}}]\label{magic}
Let $\{\mu_j\}_{j=0}^\infty$ be a sequence of complex numbers, and $e_k({\bf x})$ be the $k$-th elementary symmetric function in the variables $x_1,\dots,x_n$.  Then
\begin{equation}\label{more-beauty}
\sum_{i\le j} \mu_{j-i}e_i({\bf x})e_j({\bf x}) = e_n({\bf x})\sum_{k=0}^n \gamma_k e_{n-k}\left({\bf x }+ \frac{1}{{\bf x}}\right)
\end{equation}
where $1/{\bf x}=(1/x_1,\dots,1/x_n)$ and
\begin{equation}\label{gamma}
\gamma_k = \sum_{j=0}^{\lfloor k/2\rfloor} \binom{k}{j} \mu_{k-2j}.
\end{equation}
\end{theorem}

Note the sum on the left hand side of \eqref{more-beauty} runs over the available indices of the elementary symmetric functions, $0\le i\le j \le n$.

  
For the remainder of the section let $\PPP_n^+$ denote the set of all real polynomials of degree at most $n$, whose zeros are all real and non-positive, and let $\PPP^+=\cup_{n=0}^\infty \PPP_n^+$. By convention we also let $0 \in \PPP_n^+$ for all $n \ge 0$. We say that a polynomial $f \in \CC[x]$ is stable if either 
$f \equiv 0$ or 
$$
f(z) = 0 \ \ \ \mbox{ implies } \ \ \ \Im \ z \leq 0.
$$

The following characterizations follow from a straightforward generalization of the proof of Theorem \ref{main}.  Below, we frequently use the algebraic relations
\begin{equation}\label{even}
\sum_{i\le j} \mu_{j-i}e_je_ic^{i+j} = \sum_{k=0}^\infty\left(\sum_{j\ge 0} \alpha_je_{k+j}e_{k-j} \right) c^{2k}
\end{equation}
when $\mu=\{\alpha_0,0,\alpha_1,0,\alpha_2,0,\dots\}$, and
\begin{equation}\label{odd}
\sum_{i\le j} \mu_{j-i}e_je_ic^{i+j} = \sum_{k=0}^\infty\left(\sum_{j\ge 0} \alpha_je_{k+j+1}e_{k-j} \right) c^{2k+1}
\end{equation}
when $\mu=\{0,\alpha_0,0,\alpha_1,0,\alpha_2,\dots\}$.

\begin{theorem}\label{nonlinear-real-roots1}

Let $n \in \NN$, and let $a,b,c,r$ and $s$ be real numbers with $a < b$, $c >0$ and $r,s \geq 0$.  Let further $\{\alpha_k\}_{k=0}^\infty$ be a sequence of real numbers.  The following are equivalent:

\begin{enumerate}\itemsep 6pt
\item \label{Qn}
$ \displaystyle
Q_n^\alpha(x) := \sum_{k=0}^{\lfloor n/2 \rfloor} \left( \sum_{j=0}^k \frac{\alpha_j}{(k+j)!(k-j)!} \right) \frac{x^k}{(n-2k)!} \in \PPP_n^+, 
$

\item \label{Ya}
$ \displaystyle
Y_n^\alpha(x) := \sum_{k=0}^n \left(\sum_{j=0}^k \alpha_j \binom{n}{k+j}\binom{n}{k-j}\right) x^k \in \PPP_n^+,
$

\item\label{Tabc}
$ \displaystyle
U_\alpha^{(a,b,c)}(g) := \sum_{k=0}^\infty \left(\sum_{j\ge 0} \alpha_j\frac{g^{(k+j)}(x)}{(k+j)!}\frac{g^{(k-j)}(x)}{(k-j)!} \right) c^k(x-a)^k(x-b)^k
$\\
has zeros only in $[a,b]$ for all polynomials $g$ of degree at most $n$ whose zeros all lie in the interval $[a,b]$,
\item \label{Tabc-neg}
$U_\alpha^{(a,b,-c)}(g)$ has zeros only in  $(-\infty,a]\cup [b,\infty)$, for all polynomials $g$ of degree at most $n$ whose zeros lie in $(-\infty,a]$.
\item \label{Tabc2-neg}
$U_\alpha^{(a,b,-c)}(g)$ has zeros only in  $(-\infty,a]\cup [b,\infty)$, for all polynomials $g$ of degree at most $n$ whose zeros lie in $[b,\infty)$. 
\item \label{Ede} 
$ \displaystyle
E_\alpha^{(r,s)}(g):=\sum_{k=0}^\infty \left(\sum_{j\ge 0} \alpha_j\frac{g^{(k+j)}(x)}{(k+j)!}\frac{g^{(k-j)}(x)}{(k-j)!} \right) (rx-s)^{k}
$
has only real zeros 
for all $g\in\PPP_n^+$.
\end{enumerate}
Moreover if any of the conditions above hold and $r=0$, then $E_\alpha^{(0,s)}(g)$ is stable whenever $g$ is stable and of degree at most $n$. 
\end{theorem}
\begin{proof}
The equivalence \eqref{Ya}$\Leftrightarrow$\eqref{Qn} follows from \cite[Theorem 4.2]{B}. To prove the remainder of the theorem we will show 
\eqref{Qn}$\Rightarrow$\eqref{Tabc}$\Rightarrow$\eqref{Ya}, 
\eqref{Qn}$\Rightarrow$\eqref{Tabc-neg}$\Rightarrow$\eqref{Ya}, 
\eqref{Qn}$\Rightarrow$\eqref{Tabc2-neg}$\Rightarrow$\eqref{Ya} and \eqref{Qn}$\Rightarrow$\eqref{Ede}$\Rightarrow$\eqref{Ya}.

\eqref{Qn}$\Rightarrow$\eqref{Tabc} is proved as follows. Assume \eqref{Qn}, and 
suppose $g(x) = \prod_{i=1}^n (x-\theta_i)$.   We shall prove that $U_\alpha^{(a,b,c)}(g)(y) \neq 0$ for $y \in \CC \setminus [a,b]$. We claim that we may choose $\xi \in \CC$ such that $\xi^2= c(y-a)(y-b)$ and $\Re (\xi/(y-a)) >0$. Indeed if $\zeta$ is a square root of $c(y-a)(y-b)$ and $\Re(\zeta/(y-a)) = 0$, then 
$\zeta^2/(y-a)^2 = c(y-b)/(y-a)$ is a negative real number and thus $y \in (a,b)$. Hence we may choose $\xi = \pm \zeta$. Next we claim that $\Re(\xi/(y-\theta)) >0$ for all 
$\theta \in [a,b]$. Indeed if $\Re(\xi/(y-\theta)) \leq 0$ for some $a\leq \theta \leq b$, then $\Re(\xi/(y-\theta')) = 0$ for some $a\leq \theta' \leq b$ (since $\Re (\xi/(y-a)) >0$). Hence 
$$
\frac {\xi^2}{(y-\theta')^2}= \frac {c(y-a)(y-b)}{(y-\theta')^2}  
$$
is a negative real number. This implies $y \in [a,b]$, which is a contradiction, and proves the claim. 

We now suppose that $U_\alpha^{(a,b,c)}(g)(y) = 0$, for some $y\in \CC \setminus [a,b]$, and obtain a contradiction.  Let $\mu=\{\alpha_0,0,\alpha_1,0,\alpha_2,0,\dots\}$ and
let $\zz= (z_1,\ldots, z_n)$, where $z_i= \xi/(y-\theta_i)$ for $1 \leq i \leq n$. Thus $\Re(z_i) >0$ for all $1\leq i \leq n$. Let $e_k(\xx)$ be the $k$th elementary symmetric function in the variables $\xx =(x_1,\ldots, x_n)$. 
Then with the relation (cf. \ref{symm-fn-ident})
\[
\xi^k \frac{g^{(k)}(y)}{k!} = g(y)e_k({\bf z})  
\]
and Theorem \ref{magic},
\begin{align}
U_\alpha^{(a,b,c)}(g)(y) &=  g^2(y)\sum_{j\le k} \mu_{k-j} e_k({\bf z})e_j({\bf z}) \nonumber\\
				   &= g^2(y) e_n({\bf z}) \sum_{k=0}^n \gamma_k e_{n-k}({\bf z} + 1/{\bf z}) \nonumber\\
				   &= g^2(y) e_n({\bf z}) \sum_{k=0}^{\lfloor n/2 \rfloor} \gamma_{2k} e_{n-2k}({\bf z} + 1/{\bf z}) =0 \label{a-last}
\end{align}
Applying Theorem \ref{GWS}, there is a number $\eta$, with $\Re(\eta)>0$, such that
\[
\sum_{k=0}^{\lfloor n/2 \rfloor} \gamma_{2k} e_{n-2k}(\eta,\dots,\eta)= n!\eta^n Q_n^\alpha(\eta^{-2}) =0.
\] 
Since $\eta^{-2}$ is not a negative real number, and since $Q_n^\alpha \in\PPP_n^+$  we obtain the desired contradiction. This proves \eqref{Tabc}. 

\eqref{Tabc}$\Rightarrow$\eqref{Ya} is shown as follows.
If $U_\alpha^{(a,b,c)}(g)$ has zeros only in the interval $[a,b]$ for all $g$ of degree $n$ with zeros only in $[a,b]$, then 
\begin{align}
U_\alpha^{(a,b,c)}((x-b)^n) &= (x-b)^{2n} \sum_{k=0}^n \left(\sum_{j=0}^k \alpha_j \binom{n}{k+j}\binom{n}{k-j}\right) c^k \left(\frac{x-a}{x-b}\right)^k\nonumber\\
&= (x-b)^{2n} Y_n^\alpha\left(c\left(\frac{x-a}{x-b}\right)\right),\nonumber
\end{align}
so the zeros of $Y_n^\alpha$ must be real and non-positive.

We now sketch the proof of \eqref{Qn}$\Rightarrow$\eqref{Tabc-neg} which is almost the same as the proof for \eqref{Qn}$\Rightarrow$\eqref{Tabc}.  One must establish the claims
\begin{itemize}
\item[(I)] For $y\in\CC\setminus((-\infty,a]\cup[b,\infty))$, there is a $\xi$ such that $\xi^2=-c(y-a)(y-b)$ and $\Re(\xi/(y-a))>0$, and
\item[(II)] For $\xi$ chosen as in claim (i), $\Re(\xi/(y-\theta))>0$ for all $\theta\in(-\infty,a]$.
\end{itemize}
These follow by the same reasoning used to show \eqref{Qn}$\Rightarrow$\eqref{Tabc}.
The proof of \eqref{Qn}$\Rightarrow$\eqref{Tabc-neg} then continues as in the case for \eqref{Qn}$\Rightarrow$\eqref{Tabc}. 

The implication \eqref{Tabc-neg}$\Rightarrow$\eqref{Ya} is proved in a way similar to \eqref{Tabc}$\Rightarrow$\eqref{Ya}: 
\begin{center}
$U_\alpha^{(a,b,-c)}((x-a)^n)$ has zeros only in $(-\infty,a]\cup[b,\infty)$
\end{center} 
implies that $Y_n^\alpha$ as given by \eqref{Ya}, has only real negative zeros. 

The proofs of \eqref{Qn}$\Rightarrow$\eqref{Tabc2-neg}$\Rightarrow$\eqref{Ya} and \eqref{Qn}$\Rightarrow$\eqref{Ede}$\Rightarrow$\eqref{Ya}  follow similarly. 

For the final statement, let $g(x)= \prod_{j=1}^n (x - \zeta_j)$, where $\Im \ \zeta_j \leq 0$ for all $1\leq j \leq n$, be a stable polynomial. For $y \in \CC$ with $\Im \ y >0$, let $z_j := i\sqrt{s}/(y-\zeta_j)$. Then $\Re \ z_j >0$, and the proof follows just as above. 
\end{proof}

The \emph{Laguerre--P\'olya class}, $\LP$, of entire functions consists of all entire functions that are limits, uniformly on compact subsets of $\CC$,  of real polynomials with only real zeros. A function $\phi$ is in $\LP$ if and only if it can be expressed in the form 
$$
\phi(z)= C z^n e^{-az^2 +bz}\prod_{j=0}^\infty (1+\rho_jz)e^{-\rho_j z},
$$
where $n \in \NN$, $a,b,c \in \RR$, $a \geq 0$, and $\{\rho_j\}_{j=0}^\infty \subset \RR$ satisfies $\sum_{j=0}^\infty \rho_j^2< \infty$, see \cite[Chapter VIII]{Le}. 

$\LP^+$ consists of those functions in the Laguerre--P\'olya class that have non-negative Taylor coefficients. A function $\phi$ is in $\LP^+$ if and only it can be expressed as 
$$
\phi(z) = Cz^M e^{az}\prod_{j=0}^\infty (1+\rho_jz), 
$$
where $a, C\geq 0, M \in \NN$ and $\sum_{j=0}^\infty \rho_j < \infty$, see \cite[Chapter VIII]{Le}.

The following unbounded degree version of Theorem~\ref{nonlinear-real-roots1} is proved exactly as \cite[Theorem 5.7]{B}. 

\begin{theorem}\label{transa}

Let $a,b,c,r$ and $s$ be real numbers with $a < b$, $c >0$ and $r,s \geq 0$.  Let further $\{\alpha_k\}_{k=0}^\infty$ be a sequence of real numbers.  The following are equivalent:

\begin{enumerate}\itemsep 6pt
\item 
$ \displaystyle
Q_\infty^\alpha(x) := \sum_{k=0}^{\infty} \left( \sum_{j=0}^k \frac{\alpha_j}{(k+j)!(k-j)!} \right) {x^k} \in \LP^+, 
$

\item
$ 
U_\alpha^{(a,b,c)}(g)
$
has zeros only in $[a,b]$ for all polynomials $g$ whose zeros all lie in the interval $[a,b]$,
\item 
$U_\alpha^{(a,b,-c)}(g)$ has zeros only in  $(-\infty,a]\cup [b,\infty)$, for all polynomials $g$  whose zeros lie in $(-\infty,a]$.
\item \label{Tabc-neg2}
$U_\alpha^{(a,b,-c)}(g)$ has zeros only in  $(-\infty,a]\cup [b,\infty)$, for all polynomials $g$  whose zeros lie in $[b,\infty)$. 
\item 
$ 
E_\alpha^{(r,s)}(g)
$
has only real zeros 
for all $g\in\PPP^+$.
\end{enumerate}
Moreover if any of the conditions above hold and $r=0$, then $E_\alpha^{(0,s)}(g)$ is stable whenever $g$ is stable. 
\end{theorem}

One may also apply the operators above to functions in the Laguerre--P\'olya class. Using the methods in \cite{B}, the following transcendental extension of Theorem \ref{transa} follows.  Let $\LP_\CC$ be the class of entire functions that are limits, uniformly on compact subsets of $\CC$,  of stable polynomials. If $I \subseteq \RR$, let $\LP(I)$ be the class of real entire functions that are limits, uniformly on compact subsets of $\CC$,  of real polynomials with zeros only in $I$. 

\begin{theorem}\label{transtransa}

Let $a,b,c,r$ and $s$ be real numbers with $a < b$, $c >0$ and $r,s \geq 0$.  Let further $\{\alpha_k\}_{k=0}^\infty$ be a sequence of real numbers.  The following are equivalent:

\begin{enumerate}\itemsep 6pt
\item 
$ 
Q_\infty^\alpha(x) \in \LP^+, 
$
\item 
$U_\alpha^{(a,b,-c)}(g)$ is a function in $\LP((-\infty,a]\cup [b,\infty))$, for all functions $g \in \LP((-\infty,a])$.
\item 
$U_\alpha^{(a,b,-c)}(g)$  is a function in $\LP((-\infty,a]\cup [b,\infty))$, for all functions $g \in \LP([b,\infty))$.
\item 
$ 
E_\alpha^{(r,s)}(g)
$
is a function in $\LP$ 
for all $g\in\LP^+$.
\end{enumerate}
Moreover if any of the conditions above hold and $r=0$, then $E_\alpha^{(0,s)}(g) \in \LP_\CC$ whenever $g \in \LP_\CC$. 
\end{theorem}

\begin{remark}
As in \cite{B} one may also prove results analogous to Theorems \ref{nonlinear-real-roots1}, \ref{transa}, and \ref{transtransa}, for operators of the form 
$$
\sum_{k=0}^\infty \left(\sum_{j\ge 0} \alpha_j\frac{g^{(k+1+j)}(x)}{(k+1+j)!}\frac{g^{(k-j)}(x)}{(k-j)!} \right) P(x)^k,
$$
where $P(x)=c(x-a)(x-b)$ or $P(x)= rx-s$ with real constants $a,b,c,r,$ and $s$. 
\end{remark}

\begin{example}
By Theorem \ref{transtransa}, $E_\alpha^{(0,-s)}$ maps $\LP_\CC$ to $\LP_\CC$, and hence also $\LP$ to $\LP$, for all $s \geq 0$. 
The \emph{Hermite polynomials} \cite[p.~189]{Ra} are given by the Rodrigues--type formula
\[
H_n(x) = (-1)^ne^{x^2}\frac{d^n}{dx^n}e^{-x^2}.
\]  With $g=e^{-x^2}$, $\alpha=\{1,-1,0,0,\dots\}$, and our observation about $E_\alpha^{(0,-s)}$,
\begin{equation}\label{hermet}
E_\alpha^{(0,-s)}(g) = e^{-2x^2}\sum_{k=0}^\infty \left( \frac {H_{k}(x)} {k!} \frac {H_{k}(x)} {k!} -\frac {H_{k-1}(x)} {(k-1)!} \frac {H_{k+1}(x)} {(k+1)!}\right) (-s)^{k}\in\LP
\end{equation}
for all $s\ge 0$. We can transform this expression with the identity 
\[
(H_k(x))^2-H_{k-1}(x)H_{k+1}(x) = (k-1)!\sum_{j=0}^{k-1} \frac{2^{k-j}}{j!} H^2_{j}(x),
\]
which follows from the Christoffel--Darboux formula \cite[p.~154]{Ra}, and the Appell property of the Hermite polynomials, $H'_n(x)=2nH_{n-1}(x)$ \cite[p.~188]{Ra}.  A resulting statement equivalent to \eqref{hermet} is
\[
e^{-2x^2}\sum_{k=0}^\infty \left(\sum_{j=1}^{k}\frac{2^{j-1}}{j!} H^2_j(x)\right) \frac{(-s)^{k}}{(k+1)!}\in\LP,
\]
for all $s\ge 0$.
\end{example}

\end{document}